\documentclass[11pt]{amsart}
\usepackage{amsmath}
\usepackage[active]{srcltx}
\usepackage{t1enc}
\usepackage[latin2]{inputenc}
\usepackage{verbatim}
\usepackage{amsmath,amsfonts,amssymb,amsthm}
\usepackage[mathcal]{eucal}
\usepackage{enumerate}
\usepackage[centertags]{amsmath}
\usepackage{graphics}
\usepackage[active]{srcltx}

\newtheorem{theorem}{Theorem}

\newtheorem{lemma}{Lemma}

\newtheorem{corollary}{Corollary}

\begin{document}
\author{Gy\"orgy G\'at and Ushangi Goginava}
\title[Triangular Fej\'er Summability]{Triangular Fej\'er Summability of
Two-Dimensional Walsh-Fourier series}
\address{G. G\'at, Institute of Mathematics and Computer Science, College of
Ny\'\i regyh\'aza, P.O. Box 166, Nyiregyh\'aza, H-4400 Hungary }
\email{gatgy@nyf.hu}
\address{U. Goginava, Department of Mathematics, Faculty of Exact and
Natural Sciences, Tbilisi State University, Chavcha\-vadze str. 1, Tbilisi
0128, Georgia}
\email{zazagoginava@gmail.com}

\date{}
\maketitle

\begin{abstract}
It is proved that the operators $\sigma _{n}^{\bigtriangleup }$ of the
triangular-Fej{é}r-means of a two-dimensional Walsh--Fourier series are
uniformly bounded from the dyadic Hardy space $H_{p}$ to $L_{p}$ for all $%
4/5<p\leq \infty $.
\end{abstract}

\footnotetext{%
2010 Mathematics Subject Classification 42C10 .
\par
Research was supported by project T\'AMOP-4.2.2.A-11/1/KONV-2012-0051 and  by Shota Rustaveli National Science Foundation
grant no.13/06 (Geometry of function spaces, interpolation and embedding
theorems)}

\section{Introduction}

Lebesgue's \cite{leb} theorem is well known for trigonometric Fourier
series: the Fej{é}r means $\sigma _{n}f$ of $f$ converge to $f$ almost
everywhere if $f\in L_{1}(\mathbb{T}),\mathbb{T}:=[-\pi ,\pi )$ (see also
Zygmund \cite{zy}).

An analogous result for Walsh--Fourier series is due to Fine \cite{Fi}.
Later, Schipp \cite{Sc} showed that the maximal operator $\sigma ^{*}$ of
the Fejér means of the one-dimensional Walsh--Fourier series is of weak type
(1,1), from which the a.e. convergence follows by standard arguments.
Schipp's result implies by interpolation also the boundedness of $\sigma
^{*}:L_{p}\left( G\right) \rightarrow L_{p}\left( G\right)$, where $1<p\leq
\infty$. This fails to hold for $p=1$, but Fujii \cite{Fu} proved that $%
\sigma ^{*}$ is bounded from the dyadic Hardy space $H_{1}\left( G\right) $
to the space $L_{1}\left( G\right) $ (see also Simon \cite{Si1}). Fujii's
theorem was extended by Weisz \cite{We2}. Namely, he proved that $\sigma
^{*} $ is bounded from the martingale Hardy space $H_{p}\left( G\right) $ to
the space $L_{p}\left(G\right) $ for $p>1/2$. Simon \cite{Si2} gave a
counterexample, which shows that this boundedness does not hold for $%
0<p<1/2. $ In the endpoint case $p=1/2$, Weisz \cite{We4} proved that $%
\sigma ^{*}$ is bounded from the Hardy space $H_{1/2}\left( G\right) $ to
the space weak-$L_{1/2}\left( G\right) $. Goginava proved in \cite{GoEJA}
that the maximal operator of the Fej{\'e}r means of the one dimensional
Walsh--Fourier series is not bounded from the Hardy space $H_{1/2}\left(
G\right) $ to the space $L_{1/2}\left( G\right) $.

Marcinkievicz \cite{ma} verified for two-dimensional trigonometric Fourier
series that the Mar\-cinkiewicz-Fej{é}r means
\begin{equation*}
\sigma _{n}^{\Box }f=\frac{1}{n}\sum\limits_{j=0}^{n-1}S_{j}^{\Box }\left(
f\right)
\end{equation*}%
of a function $f\in L\log L(\mathbb{T}\times \mathbb{T})$ converge a.e. to $%
f $ as $n\rightarrow \infty $, where $S_{j}^{\Box }\left( f\right) $ denotes
the qubical partial sums of the Fourier series of $f$. Later Zhizhiashvili
\cite{zi2,zi1} extended this result to all $f\in L_{1}(\mathbb{T}\times
\mathbb{T})$.

An analogous result for two-dimensional Walsh--Fourier series is due to
Weisz \cite{We3}. Moreover, he proved that the maximal operator $\sigma
_{\ast }^{\Box }f=\sup\limits_{n\geq 1}\left\vert \sigma _{n}^{\Box
}f\right\vert $ is bounded from the dyadic martingale Hardy space $%
H_{p}\left( G\times G\right) $ to the space $L_{p}\left( G\times G\right) $
for $p>2/3$. The second author \cite{GoJAT, GoEJA} proved that the maximal
operator $\sigma _{\ast }^{\Box }$ is bounded from $H_{2/3}\left( G\times
G\right) $ to weak$-L_{2/3}\left( G\times G\right) $ and is not bounded from
$H_{2/3}\left( G\times G\right) $ to $L_{2/3}\left( G\times G\right).$

Weisz \cite{wel1-fs1,WeBud} studied the triangular partial sums and the Fej{é%
}r means
\begin{equation*}
\sigma _{{n}}^{\bigtriangleup }f=\frac{1}{{n}}\sum\limits_{j=0}^{{n}%
-1}S_{j}^{\bigtriangleup }f
\end{equation*}%
of the two-dimensional trigonometric Fourier series. This summability method
is rarely investigated in the literature (see the references in \cite%
{wel1-fs1}). In \cite{GoWe} it is proved that the maximal operator $\sigma
_{\#}^{\bigtriangleup }:=\sup\limits_{n}\left\vert \sigma _{{2}%
^{n}}^{\bigtriangleup }f\right\vert $ of the Fejér means of the triangular
partial sums of the double Walsh--Fourier series is bounded from the dyadic
Hardy space $H_{p}\left( G\times G\right) $ to the $L_{p}\left( G\times
G\right) $ if $p>1/2$, is bounded from $H_{1/2}\left( G\times G\right) $ to
the space weak- $L_{1/2}\left( G\times G\right) $ and it is not bounded from
$H_{1/2}\left( G\times G\right) $ to $L_{1/2}\left( G\times G\right) $.

For triangular partial sums it is well-known \cite{S-W-S} the operatos $%
S_{2^{A}}^{\triangle }$ are not unoformly bounded on $L_{p}$ for $1\leq
p\neq 2$.

It is proved that the operators $\sigma _{n}^{\bigtriangleup }$ of the
triangular-Fej{é}r-means of a two-dimensional Walsh--Fourier series are
uniformly bounded from the dyadic Hardy space $H_{p}$ to $L_{p}$ for all $%
4/5<p\leq \infty $.

The results for summability of quadratical partial sums of two-dimensional
Walsh-Fourier series can be found in \cite{GoJMA,GoSM,GGT,GGN}.

\section{Definitions and the notation}

Let $\mathbb{P}$ denote the set of positive integers, $\mathbb{N}\mathbf{:=}%
\mathbb{P}\mathbf{\cup \{}0\mathbf{\}.}$ Denote by $Z_{2}$ the discrete
cyclic group of order 2, that is $Z_{2}=\{0,1\},$ where the group operation
is the modulo 2 addition and every subset is open. A Haar measure on $Z_{2}$
is given such that the measure of a singleton is 1/2. Let $G$ be the
complete direct product of the countable infinite copies of the compact
groups $Z_{2}.$ The elements of $G$ are of the form $x=\left(
x_{0},x_{1},\dots ,x_{k},\dots \right) $ with $x_{k}\in \{0,1\}\left( k\in
\mathbb{N}\right).$ The group operation on $G$ is the coordinate-wise
addition, the measure (denoted by $\mu $) and the topology are the product
measure and topology. The compact Abelian group $G$ is called the Walsh
group. A base for the neighborhoods of $G$ can be given by
\begin{equation*}
I_{n}\left( x\right) :=\,I_{n}\left( x_{0},\dots ,x_{n-1}\right) :=\left\{
y\in G:\,y=\left( x_{0},\dots ,x_{n-1},y_{n},y_{n+1},\dots \right) \right\} ,
\end{equation*}%
where $I_{0}\left( x\right) :=G$ and $x\in G,n\in \mathbb{N}$. These sets
are called the dyadic intervals. Let $~0=\left( 0:i\in \mathbb{N}\right) \in
G$ denote the null element of $G$, $I_{n}:=I_{n}\left( 0\right) $ $\left(
n\in \mathbb{N}\right) $, $\overline{I}_{n}:=G\backslash I_{n}.$ Denote
\begin{equation*}
2^{N}x:=\left( x_{N},x_{N+1},...\right) ,x\in G.
\end{equation*}

For $k\in \mathbb{N}$ and $x\in G$ denote
\begin{equation*}
r_{k}\left( x\right) :=\left( -1\right) ^{x_{k}}\,\,\,\,\,\,\left( x\in
G,\;\;k\in \mathbb{N}\right)
\end{equation*}%
the $k$-th Rademacher function. If $n\in \mathbb{N}$, then $%
n=\sum\limits_{i=0}^{\infty }n_{i}2^{i},$ where $n_{i}\in \{0,1\}$ $\left(
i\in \mathbb{N}\right) $, i.e. $n$ is expressed in the number system of base
2. Denote $\left\vert n\right\vert :=\max \{j\in \mathbf{N:}n_{j}\neq 0\}$,
that is, $2^{\left\vert n\right\vert }\leq n<2^{\left\vert n\right\vert +1}.$

The Walsh--Paley system is defined as the sequence of Walsh--Paley
functions:
\begin{equation*}
w_{n}\left( x\right) :=\prod\limits_{k=0}^{\infty }\left( r_{k}\left(
x\right) \right) ^{n_{k}}=r_{\left\vert n\right\vert }\left( x\right) \left(
-1\right) ^{\sum\limits_{k=0}^{\left\vert n\right\vert
-1}n_{k}x_{k}}\,\,\,\,\,\,\left( x\in G,\;\;n\in \mathbb{P}\right).
\end{equation*}%
The Walsh--Dirichlet kernel is defined by
\begin{equation*}
D_{n}\left( x\right) =\sum\limits_{k=0}^{n-1}w_{k}\left( x\right)
,D_{0}\left( x\right) =0.
\end{equation*}%
Recall that (\cite{S-W-S})

\begin{equation}
\begin{split}
& D_{2^{n}}\left( x\right) =\left\{
\begin{array}{ll}
2^{n} & \mbox{if}\;\;x\in I_{n}, \\
0 & \mbox{if}\;\;x\in \overline{I}_{n}
\end{array},\right. \\
&  D_n(x) =-w_{n^{(i+1)}}(x)\left(\sum_{r=0}^ {i-1}n_r2^r - n_i2^i\right) \, \mbox{for}\, \,
x\in I_i\setminus I_{i+1}.
\label{dir}
\end{split}%
\end{equation}

In this paper we consider the double system $\left\{ w_{i}\left( x\right)
w_{j}\left( y\right) :i,j\in \mathbb{N}\right\} $ on $G\times G$.

The rectangular partial sums of the 2-dimensional Walsh--Fourier series are
defined as
\begin{equation*}
S_{M,N}f(x,y):=\sum\limits_{i=0}^{M-1}\sum\limits_{j=0}^{N-1}\widehat{f}%
\left( i,j\right) w_{i}\left( x\right) w_{j}\left( y\right) ,
\end{equation*}%
where the number
\begin{equation*}
\widehat{f}\left( i,j\right) =\int_{G\times G}f\left( x,y\right) w_{i}\left(
x\right) w_{j}\left( y\right) d\mu \left( x,y\right)
\end{equation*}%
is said to be the $\left( i,j\right) $th Walsh--Fourier coefficient of the
function \thinspace $f.$ Denote
\begin{equation*}
S_{M}^{\Box }f(x,y):=S_{M,M}f(x,y).
\end{equation*}

The triangular partial sums of the 2-dimensional Walsh--Fourier series are
defined as
\begin{equation*}
S_{k}^{\bigtriangleup
}f(x,y):=\sum\limits_{i=0}^{k-1}\sum\limits_{j=0}^{k-i-1}\widehat{f}\left(
i,j\right) w_{i}\left( x\right) w_{j}\left( y\right).
\end{equation*}%
Denote
\begin{equation*}
D_{k}^{\Box }\left( x,y\right) :=D_{k}\left( x\right) D_{k}\left( y\right)
\end{equation*}%
and
\begin{equation*}
D_{k}^{\bigtriangleup }\left( x,y\right)
:=\sum\limits_{i=0}^{k-1}\sum\limits_{j=0}^{k-i-1}w_{i}\left( x\right)
w_{j}\left( y\right).
\end{equation*}

The norm (or the quasinorm) of the space $L_{p}\left( G\times G\right) $ is
defined by
\begin{equation*}
\left\Vert f\right\Vert _{p}:=\left( \int_{G\times G}\left\vert f\left(
x,y\right) \right\vert ^{p}d\mu \left( x,y\right) \right)
^{1/p}\,\,\,\,\left( 0<p\leq \infty \right) .
\end{equation*}%
The space weak-$L_{p}\left( G\times G\right) $ consists of all measurable
functions $f$ for which
\begin{equation*}
\left\Vert f\right\Vert _{\text{weak}-L_{p}\left( G\times G\right)
}:=\sup\limits_{\lambda >0}\lambda \mu \left( \left\vert f\right\vert
>\lambda \right) ^{1/p}<+\infty .
\end{equation*}

The $\sigma $-algebra generated by the dyadic 2-dimensional $I_{k}\left(
x\right) \times I_{k}\left( y\right) $ cubes of measure $2^{-k}\times 2^{-k}$
will be denoted by $F_{k}\left( k\in \mathbb{N}\right) $. Denote by $%
f=\left( f^{\left( n\right) },n\in N\right) $ one-parameter martingales with
respect to $\left( F_{n},\;n\in N\right) $ (for details see e.g. \cite{We4}%
). The maximal function of a martingale $f$ is defined by
\begin{equation*}
f^{\ast }=\sup\limits_{n\in N}\left\vert f^{\left( n\right) }\right\vert.
\end{equation*}%
In case $f\in L_{1}\left( G\times G\right) $, the maximal function can also
be given by
\begin{equation*}
f^{\ast }\left( x,y\right) =\sup\limits_{n\geq 1}\frac{1}{\mu \left(
I_{n}(x)\times I_{n}(y)\right) }\left\vert \int_{I_{n}(x)\times
I_{n}(y)}f\left( s,t\right) d\mu \left( s,t\right) \right\vert ,\,\,
\end{equation*}%
\begin{equation*}
\left( x,y\right) \in G\times G.
\end{equation*}

For $0<p<\infty $ the martingale Hardy space $H_{p}(G\times G)$ consists of
all martingales for which
\begin{equation*}
\left\| f\right\| _{H_{p}}:=\left\| f^{*}\right\| _{p}<\infty.
\end{equation*}

If $f\!\in \!L_{1}\left( G\!\times \!G\right) $, then it is easy to show
that the sequence $\left( S_{2^{n},2^{n}}\left( f\right) :n\!\in \!\mathbf{N}%
\right) $ is a martingale. If $f$ is a martingale, that is $f=(f^{\left(
0\right) },f^{\left( 1\right) },\dots )$, then the Walsh--Fourier
coefficients must be defined in a little bit different way:
\begin{equation*}
\widehat{f}\left( i,j\right) =\lim\limits_{k\rightarrow \infty
}\int_{G\times G}f^{\left( k\right) }\left( x,y\right) w_{i}\left( x\right)
w_{j}\left( y\right) d\mu \left( x,y\right).
\end{equation*}%
The Walsh--Fourier coefficients of $f\in L_{1}\left( G\times G\right) $ are
the same as the ones of the martingale $\left( S_{2^{n},2^{n}}\left(
f\right) :n\in \mathbb{N}\right) $ obtained from $f$.

For $n\in \mathbb{P}$ and a martingale $f$ the Marcinkiewicz-Fejér means and
triangular Fejér means of order $n$ of the 2-dimensional Walsh--Fourier
series of a function $f$ is given by
\begin{equation*}
\sigma _{n}^{\Box }f(x,y)=\frac{1}{n}\sum\limits_{j=0}^{n-1}S_{j}^{\Box
}f(x,y)
\end{equation*}%
and
\begin{equation*}
\sigma _{n}^{\bigtriangleup }f(x,y)=\frac{1}{n}\sum%
\limits_{j=0}^{n-1}S_{j}^{\bigtriangleup }f(x,y),
\end{equation*}%
respectively. It is easy to show that
\begin{equation*}
\sigma _{n}^{\Box }f(x,y)=\int_{G\times G}f\left( s,t\right) K_{n}^{\Box
}\left( x+s,y+t\right) d\mu \left( s,t\right)
\end{equation*}%
and
\begin{equation}
\sigma _{n}^{\bigtriangleup }f(x,y)=\int_{G\times G}f\left( s,t\right)
K_{n}^{\bigtriangleup }\left( x+s,y+t\right) d\mu \left( s,t\right) ,
\label{triang}
\end{equation}%
where
\begin{equation*}
K_{n}^{\Box }\left( x,y\right) :=\frac{1}{n}\sum\limits_{j=0}^{n-1}D_{j}^{%
\Box }\left( x,y\right)
\end{equation*}%
and
\begin{equation*}
K_{n}^{\bigtriangleup }\left( x,y\right) :=\frac{1}{n}\sum%
\limits_{j=0}^{n-1}D_{j}^{\bigtriangleup }\left( x,y\right).
\end{equation*}%
We can write%
\begin{eqnarray}
K_{n}^{\bigtriangleup }\left( x,y\right) &=&\frac{1}{n}\sum%
\limits_{k=0}^{n-1}D_{k}^{\bigtriangleup }\left( x,y\right)
\label{tr-kernel} \\
&=&\frac{1}{n}\sum\limits_{k=1}^{n-1}\sum\limits_{i=0}^{k-1}\sum%
\limits_{j=0}^{k-i-1}w_{i}\left( x\right) w_{j}\left( y\right)  \notag \\
&=&\frac{1}{n}\sum\limits_{k=1}^{n-1}\sum\limits_{i=0}^{k-1}w_{i}\left(
x\right) D_{k-i}\left( y\right)  \notag \\
&=&\frac{1}{n}\sum\limits_{k=1}^{n-1}\sum\limits_{i=1}^{k}w_{k-i}\left(
x\right) D_{i}\left( y\right)  \notag \\
&=&\frac{1}{n}\sum\limits_{i=1}^{n-1}\sum\limits_{k=i}^{n-1}w_{k-i}\left(
x\right) D_{i}\left( y\right)  \notag \\
&=&\frac{1}{n}\sum\limits_{i=1}^{n-1}D_{n-i}\left( x\right) D_{i}\left(
y\right) .  \notag
\end{eqnarray}

A bounded measurable function $a$ is a $p$-atom if there exists a
dyadic\thinspace 2-dimen\-sional cube $I$ such that

a) $\int_{I}a d\mu =0$;

b) $\left\Vert a\right\Vert _{\infty }\leq \mu (I)^{-1/p}$;

c) supp $a\subset I$.

An operator $T$ which maps the set of martingale into the collection of
measurable functions will be called $p$-quasi-local if there exists a
constant $c_{p}>0$ such that for evey $p$-atom $a$

\begin{equation*}
\int\limits_{G\times G\backslash I}|Ta|^{p}\leq c_{p}<\infty ,
\end{equation*}%
where $I$ is the support of the atom.

\section{Formulation of main results}

\begin{theorem}
\label{main}If $4/5<p\leq \infty $, then the operators $\sigma
_{n}^{\bigtriangleup }$ are uniformly bounded from the Hardy space $%
H_{p}\left( G\times G\right) $ to the space $L_{p}\left( G\times G\right) .$%
In particular, if $f\in H_{p}\left( G\times G\right) $, then%
\begin{equation*}
\sup\limits_{n}\left\Vert \sigma _{n}^{\bigtriangleup }f\right\Vert _{p}\leq c_{p}\left\Vert
f\right\Vert _{H_{p}}.
\end{equation*}
\end{theorem}

\section{Auxiliary proposition}

We shall need the following lemma (see \cite{We4}).

\begin{lemma}
\label{weisz}Suppose that the operator $T$ \thinspace is\thinspace $\sigma $%
-sublinear and $p$-quasi-local for each $0<p_{0}<p\leq 1$. If $T$ is bounded
from $L_{\infty }(G\times G)$ to $L_{\infty }(G\times G)$, then

\begin{equation*}
\left\Vert Tf\right\Vert _{p}\leq c\left( p\right) \left\Vert f\right\Vert
_{H_{p}}\,\,\,\,\,\,\,\,\,\,\,(f\in H_{p}\left( G\times G\right) )
\end{equation*}%
for every $0<p_{0}<p<\infty .\,$
\end{lemma}

\begin{lemma}
\label{sup1}Let $p\in (1/2,1].$ Then%
\begin{equation*}
\int\limits_{G}\left( \sup\limits_{1\leq n\leq 2^{N}}\left\vert
\sum\limits_{j=1}^{n}D_{j}\left( x\right) \right\vert \right) ^{p}d\mu
\left( x\right) \leq c_{p}2^{N\left( 2p-1\right) }.
\end{equation*}
\end{lemma}

The proof can be found in \cite{GogAM}.

\begin{lemma}
\label{lemma1}Let $p\in (1/2,1].$ Then%
\begin{equation*}
\int\limits_{G}\sup\limits_{1\leq n\leq 2^{N}}\left\vert
\sum\limits_{k=1}^{n}D_{k}\left( x\right) \left( n-k+1\right) \right\vert
^{p}d\mu \left( x\right) \leq c_{p}2^{N\left( 3p-1\right) }.
\end{equation*}
\end{lemma}

\begin{proof}[Proof of Lemma \protect\ref{lemma1}]
Applying Lemma \ref{sup1} and Abel's transformation we obtain%
\begin{equation*}
\int\limits_{G}\sup\limits_{1\leq n\leq 2^{N}}\left\vert
\sum\limits_{k=1}^{n}D_{k}\left( x\right) \left( n-k+1\right) \right\vert
^{p}d\mu \left( x\right)
\end{equation*}%
\begin{equation*}
\leq c_{p}2^{Np}\int\limits_{G}\sup\limits_{1\leq n\leq 2^{N}}\left\vert
\sum\limits_{k=1}^{n}D_{k}\left( x\right) \right\vert ^{p}d\mu \left(
x\right)
\end{equation*}%
\begin{equation*}
+c_{p}\int\limits_{G}\left( \sup\limits_{1\leq n\leq
2^{N}}\sum\limits_{k=1}^{n}\left\vert \sum\limits_{l=1}^{k}D_{k}\left(
x\right) \right\vert \right) ^{p}d\mu \left( x\right)
\end{equation*}%
\begin{equation*}
\leq c_{p}2^{Np}\int\limits_{G}\sup\limits_{1\leq n\leq 2^{N}}\left\vert
\sum\limits_{k=1}^{n}D_{k}\left( x\right) \right\vert ^{p}d\mu \left(
x\right) \leq c_{p}2^{N\left( 3p-1\right) }.
\end{equation*}

Lemma \ref{lemma1} is proved.
\end{proof}

\begin{lemma}
\label{sup3} Let $p\in (1/2,1].$ Then%
\begin{equation*}
\int\limits_{G}\left( \sup\limits_{0\leq q<2^{N}}\left\vert
\sum\limits_{k=q}^{2^{N}-1}D_{k}\left( x\right) \left( k-q+1\right)
\right\vert \right) ^{p}d\mu \left( x\right) \leq c_{p}2^{N\left(
3p-1\right) }.
\end{equation*}
\end{lemma}

\begin{proof}[Proof of Lemma \protect\ref{sup3}]
From Lemma \ref{sup1} we have%
\begin{equation*}
\int\limits_{G}\left( \sup\limits_{0\leq q<2^{N}}\left\vert
\sum\limits_{k=q}^{2^{N}-1}D_{k}\left( x\right) \left( k-q+1\right)
\right\vert \right) ^{p}d\mu \left( x\right)
\end{equation*}%
\begin{equation*}
\leq c_{p}2^{Np}\int\limits_{G}\sup\limits_{0\leq q<2^{N}}\left( \left\vert
\sum\limits_{k=0}^{q-1}D_{k}\left( x\right) \right\vert \right) ^{p}d\mu
\left( x\right)
\end{equation*}%
\begin{equation*}
+c_{p}\int\limits_{G}\sup\limits_{0\leq q<2^{N}}\left( \left\vert
\sum\limits_{k=0}^{q-1}\left( k+1\right) D_{k}\left( x\right) \right\vert
\right) ^{p}d\mu \left( x\right) \leq c_{p}2^{N\left( 3p-1\right) }.
\end{equation*}

Lemma \ref{sup3} is proved.
\end{proof}

Let $\alpha :=\left( \alpha _{1},\alpha _{2}\right) :\mathbb{N}%
^{2}\rightarrow \mathbb{N}^{2}$ be a function. Define the following
Marcinkiewicz-like kernels
\begin{equation*}
\frac{1}{n}\sum\limits_{k=0}^{n-1}D_{\alpha _{1}\left( n,k\right) }\left(
x\right) D_{\alpha _{2}\left( n,k\right) }\left( y\right) ,n\in \mathbb{P}.
\end{equation*}

Denote by $\#B$ the cardinality of set $B$. Suppose that
\begin{equation}
\#\left\{ l\in \mathbb{N}:\alpha _{j}\left( n,l\right) =\alpha _{j}\left(
n,k\right) ,l<n\right\} \leq C~\ \left( k<n,n\in \mathbb{P},j=1,2\right).
\label{alpa1}
\end{equation}

\begin{lemma}
\label{lemma2}Let $p\in (4/5,1]$. Then%
\begin{equation*}
\sup\limits_{1\leq n\leq 2^{N}}\int\limits_{\overline{I}_{N}\times \overline{%
I}_{N}}\left\vert \sum\limits_{k=0}^{n-1}D_{\alpha _{1}\left( n,k\right)
}\left( x\right) D_{\alpha _{2}\left( n,k\right) }\left( y\right)
\right\vert ^{p}d\mu \left( x,y\right) \leq c_{p}2^{N\left( 3p-2\right) }.
\end{equation*}
\end{lemma}

\begin{proof}[Proof of Lemma \protect\ref{lemma2}]
\ In the sequel we use termonology and methods of paper \cite{GatMar}. Set%
\begin{equation*}
n^{\left( s\right) }:=\sum\limits_{k=s}^{\infty }n_{k}2^{s},
J_{k}:=I_{k}\backslash I_{k+1}. \mbox{\ Thus \ } n^{\left(0\right)}=n,
n^{\left(N\right) }=0, n,k\in \mathbb{N}.
\end{equation*}%
Then by
\begin{equation*}
\sum\limits_{k=0}^{n-1}D_{\alpha _{1}\left( n,k\right) }\left( x\right)
D_{\alpha _{2}\left( n,k\right) }\left( y\right)  = \sum\limits_{s=0}^{N-1}
n_s \sum\limits_{k=0}^{2^s-1}
D_{\alpha _{1} (n,k+n^{( s+1) } )}(x)D_{\alpha _{1} (n,k+n^{( s+1) } )}(y)
\end{equation*}%

and $\overline{I}_N=\cup_{i=0}^{N-1}J_i = \cup_{i=0}^{N-1}\left(I_i \setminus I_{i+1}\right) $
we have%
\begin{equation}
\int\limits_{\overline{I}_{N}\times \overline{I}_{N}}\left\vert
\sum\limits_{k=0}^{n-1}D_{\alpha _{1}\left( n,k\right) }\left( x\right)
D_{\alpha _{2}\left( n,k\right) }\left( y\right) \right\vert ^{p}d\mu \left(
x,y\right)  \label{A+B}
\end{equation}%
\begin{equation*}
\leq \sum\limits_{s=0}^{N-1}\int\limits_{\overline{I}_{N}\times \overline{I}%
_{N}}\left\vert \sum\limits_{k=0}^{2^{s}-1}D_{\alpha _{1}\left(
n,k+n^{\left( s+1\right) }\right) }\right.
\end{equation*}%
\begin{equation*}
\left. \times D_{\alpha _{2}\left( n,k+n^{\left( s+1\right) }\right) }\left(
y\right) \right\vert ^{p}d\mu \left( x,y\right)
\end{equation*}%
\begin{equation*}
=\sum\limits_{i=0}^{N-1}\sum\limits_{j=i}^{N-1}\sum\limits_{s=0}^{N-1}\int%
\limits_{J_{i}\times J_{j}}\left\vert \sum\limits_{k=0}^{2^{s}-1}D_{\alpha
_{1}\left( n,k+n^{\left( s+1\right) }\right) }\left( x\right) \right.
\end{equation*}%
\begin{equation*}
\left. \times D_{\alpha _{2}\left( n,k+n^{\left( s+1\right) }\right) }\left(
y\right) \right\vert ^{p}d\mu \left( x,y\right)
\end{equation*}%
\begin{equation*}
+\sum\limits_{i=0}^{N-1}\sum\limits_{j=0}^{i-1}\sum\limits_{s=0}^{N-1}\int%
\limits_{J_{i}\times J_{j}}\left\vert \sum\limits_{k=0}^{2^{s}-1}D_{\alpha
_{1}\left( n,k+n^{\left( s+1\right) }\right) }\left( x\right) \right.
\end{equation*}%
\begin{equation*}
\left. \times D_{\alpha _{2}\left( n,k+n^{\left( s+1\right) }\right) }\left(
y\right) \right\vert ^{p}d\mu \left( x,y\right) =:A+B,
\end{equation*}
where $n\leq 2^{N}$. We can write%
\begin{equation}
A=\sum\limits_{i=0}^{N-1}\sum\limits_{j=i}^{N-1}\sum\limits_{s=0}^{i-1}\int%
\limits_{J_{i}\times J_{j}}\left\vert \sum\limits_{k=0}^{2^{s}-1}D_{\alpha
_{1}\left( n,k+n^{\left( s+1\right) }\right) }\left( x\right) \right.
\label{A1+A2}
\end{equation}%
\begin{equation*}
\left. \times D_{\alpha _{2}\left( n,k+n^{\left( s+1\right) }\right) }\left(
y\right) \right\vert ^{p}d\mu \left( x,y\right)
\end{equation*}%
\begin{eqnarray*}
&&+\sum\limits_{i=0}^{N-1}\sum\limits_{j=i}^{N-1}\sum\limits_{s=i}^{N-1}\int%
\limits_{J_{i}\times J_{j}}\left\vert \sum\limits_{k=0}^{2^{s}-1}D_{\alpha
_{1}\left( n,k+n^{\left( s+1\right) }\right) }\left( x\right) \right. \\
&&\left. \times D_{\alpha _{2}\left( n,k+n^{\left( s+1\right) }\right)
}\left( y\right) \right\vert ^{p}d\mu \left( x,y\right) \\
&&=:A_{1}+A_{2}.
\end{eqnarray*}

Since for every $n\in\mathbb N$ and $x\in J_i = I_i\setminus I_{i+1}$ we have $\vert D_n(x)\vert \le c2^i$ then we also have
\begin{equation*}
\left\vert D_{\alpha _{1}\left( n,k\right) }\left( x\right) D_{\alpha
_{2}\left( n,k\right) }\left( y\right) \right\vert \le c2^{i+j},\left(
x,y\right) \in J_{i}\times J_{j}.
\end{equation*}%
For $A_{1}$ this implies in the case of $1>\left( p>2/3\right) $%
\begin{equation*}
A_{1}\leq
c_{p}\sum\limits_{i=0}^{N}\sum\limits_{j=i}^{N}\sum\limits_{s=0}^{i-1}\frac{%
2^{p\left( s+i+j\right) }}{2^{i+j}}\leq c_{p}2^{N\left( 3p-2\right) }.
\end{equation*}

If $p=1$, then for $A_1$ we have

\[
\begin{split}
 A_{1} & \leq
c_{1}\sum\limits_{i=0}^{N}\sum\limits_{j=i}^{N}\sum\limits_{s=0}^{i-1}\frac{%
2^{s+i+j }}{2^{i+j}}  \le c_{1}\sum\limits_{i=0}^{N}\sum\limits_{j=i}^{N} 2^i \\
& \le c_{1}\sum\limits_{i=0}^{N} (N-i+1)2^i \le c_12^N\leq c_{p}2^{N\left( 3p-2\right) }.
\end{split}
\]

That is, for every $\left( p>2/3\right)$
\begin{equation}
A_{1}\leq
c_{p}\sum\limits_{i=0}^{N}\sum\limits_{j=i}^{N}\sum\limits_{s=0}^{i-1}\frac{%
2^{p\left( s+i+j\right) }}{2^{i+j}}\leq c_{p}2^{N\left( 3p-2\right) }.
\label{A1}
\end{equation}

Turn our attention to $A_{2}$. By (\ref{dir}) for $x\in J_{i}=I_i\setminus I_{i+1}$ we can
write%
\begin{equation*}
D_{\alpha _{1}(n,k+n^{\left( s+1\right) })}\left( x\right) =-w_{\left(
\alpha _{1}\left( n,k+n^{\left( s+1\right) }\right) \right) ^{\left(
i+1\right) }}\left( x\right)
\end{equation*}%
\begin{equation*}
\times \left( \sum\limits_{r=0}^{i-1}\left( \alpha _{1}\left( n,k+n^{\left(
s+1\right) }\right) \right) _{r}2^{r}-\left( \alpha _{1}\left( n,k+n^{\left(
s+1\right) }\right) \right) _{i}2^{i}\right).
\end{equation*}

Set%
\begin{equation*}
\theta \left( s,k,i\right) :=\sum\limits_{r=0}^{i-1}\left( \alpha _{1}\left(
n,k+n^{\left( s+1\right) }\right) \right) _{r}2^{r}-\left( \alpha _{1}\left(
n,k+n^{\left( s+1\right) }\right) \right) _{i}2^{i}.
\end{equation*}

Apply the Cauchy-Schwarz inequality and follow the method of \cite{GatMar}:%
\begin{equation*}
\int\limits_{J_{i}\times J_{j}}\left\vert
\sum\limits_{k=0}^{2^{s}-1}D_{\alpha _{1}\left( n,k+n^{\left( s+1\right)
}\right) }\left( x\right) D_{\alpha _{2}\left( n,k+n^{\left( s+1\right)
}\right) }\left( y\right) \right\vert ^{p}d\mu \left( x,y\right)
\end{equation*}%
\begin{eqnarray*}
&\leq &\int\limits_{J_{j}}2^{-i\left( 1-p/2\right) }\Biggl( %
\int\limits_{J_{i}}\Biggl\vert \sum\limits_{k=0}^{2^{s}-1}D_{\alpha
_{1}\left( n,k+n^{\left( s+1\right) }\right) }\left( x\right) \\
&& \times D_{\alpha _{2}\left( n,k+n^{\left( s+1\right) }\right) }\left(
y\right) \Biggr\vert ^{2}d\mu \left( x\right) \Biggr) ^{p/2}d\mu \left(
y\right)
\end{eqnarray*}%
\begin{equation*}
=\int\limits_{J_{j}}2^{-i\left( 1-p/2\right) }\Biggl( \int\limits_{J_{i}}%
\sum\limits_{k,l=0}^{2^{s}-1}D_{\alpha _{1}\left( n,k+n^{\left( s+1\right)
}\right) }\left( x\right)
\end{equation*}%
\begin{equation*}
\times D_{\alpha _{1}\left( n,l+n^{\left( s+1\right) }\right) }\left(
x\right) D_{\alpha _{2}\left( n,k+n^{\left( s+1\right) }\right) }\left(
y\right) D_{\alpha _{2}\left( n,l+n^{\left( s+1\right) }\right) }\left(
y\right) d\mu \left( x\right) \Biggr) ^{p/2}d\mu \left( y\right)
\end{equation*}%
\begin{equation*}
=\int\limits_{J_{j}}2^{-i\left( 1-p/2\right) }\Biggl( \int\limits_{J_{i}}%
\sum\limits_{k,l=0}^{2^{s}-1}w_{\left( \alpha _{1}\left( n,k+n^{\left(
s+1\right) }\right) \right) ^{\left( i+1\right) }}\left( x\right)
\end{equation*}%
\begin{equation*}
\times w_{\left( \alpha _{1}\left( n,l+n^{\left( s+1\right) }\right) \right)
^{\left( i+1\right) }}\left( x\right) \theta \left( s,k,i\right) \theta
\left( s,l,i\right)
\end{equation*}%
\begin{equation*}
\times D_{\alpha _{2}\left( n,k+n^{\left( s+1\right) }\right) }\left(
y\right) D_{\alpha _{2}\left( n,l+n^{\left( s+1\right) }\right) }\left(
y\right) d\mu \left( x\right) \Biggr) ^{p/2}d\mu \left( y\right)
\end{equation*}%
\begin{equation*}
=\int\limits_{J_{j}}2^{-i\left( 1-p/2\right) }\Biggl( \sum%
\limits_{k,l=0}^{2^{s}-1}\theta \left( s,k,i\right) \theta \left(
s,l,i\right)
\end{equation*}%
\begin{equation*}
\times D_{\alpha _{2}\left( n,k+n^{\left( s+1\right) }\right) }\left(
y\right) D_{\alpha _{2}\left( n,l+n^{\left( s+1\right) }\right) }\left(
y\right) \int\limits_{J_{i}}w_{\left( \alpha _{1}\left( n,k+n^{\left(
s+1\right) }\right) \right) ^{\left( i+1\right) }}\left( x\right)
\end{equation*}%
\begin{equation*}
\times w_{\left( \alpha _{1}\left( n,l+n^{\left( s+1\right) }\right) \right)
^{\left( i+1\right) }}\left( x\right) d\mu \left( x\right) \Biggr) %
^{p/2}d\mu \left( y\right).
\end{equation*}

Discuss the integral on the set $J_i=I_i\setminus I_{i+1}$%
\begin{equation}
\int\limits_{J_{i}}w_{( \alpha _{1}( n,k+n^{\left( s+1\right) }) )^{\left(
i+1\right) }}\left( x\right) w_{(\alpha _{1}(n,l+n^{\left( s+1\right)
}))^{\left( i+1\right) }}\left( x\right) d\mu \left( x\right).
\end{equation}\label{ortog}%
If it differs from zero, then the $i+1$-th, $i+2$-th, ... coordinates of $%
\alpha _{1}(n,n^{\left( s+1\right) }+k)$ and $\alpha _{1}(n,n^{\left(
s+1\right) }+l)$ should be equal.

We have that for every $k$ there
exists only a bounded numbers of $l$'s for which $\alpha _{1}(n,n^{\left(
s+1\right) }+k)=\alpha _{1}(n,n^{\left( s+1\right) }+l)$. These facts give
that for every $k$ there exists - at most - $C2^{i}$ number of $l$'s for
which this integral is not zero. This will be very important in the estimation of
$A_2$ because - at first sight - $k$ and $l$ are elements of $\{0,1,\dots, 2^s-1\}$ and consequently
this would mean $2^{2s}$ addends. But this is not the case, because for every $k\in \{0,1,\dots, 2^s-1\}$ the number of $l$'s to be taken
is $2^i$ because we need to take the integrals only when they are not zero. Thus, the number of $(k,l)$ pairs we have to take account is only $2^{s+i}$ and not $2^{2s}$.

Consequently, apply the facts that $\left\vert \theta \left( s,k,i\right)
\right\vert \leq c2^{i},\left\vert D_{\alpha _{2}\left( n,k+n^{\left(
s+1\right) }\right) }\left( y\right) \right\vert \leq c2^{j}$
we obtain

\begin{equation*}
\begin{split}
&\int\limits_{J_{j}}2^{-i\left( 1-p/2\right) }\Biggl( \sum%
\limits_{k,l=0}^{2^{s}-1}\theta \left( s,k,i\right) \theta \left(
s,l,i\right)\\
& \times D_{\alpha _{2}\left( n,k+n^{\left( s+1\right) }\right) }\left(
y\right) D_{\alpha _{2}\left( n,l+n^{\left( s+1\right) }\right) }\left(
y\right) \int\limits_{J_{i}}w_{\left( \alpha _{1}\left( n,k+n^{\left(
s+1\right) }\right) \right) ^{\left( i+1\right) }}\left( x\right)
\\
& \times w_{\left( \alpha _{1}\left( n,l+n^{\left( s+1\right) }\right) \right)
^{\left( i+1\right) }}\left( x\right) d\mu \left( x\right) \Biggr)
^{p/2}d\mu \left( y\right)\\
& \le c\int\limits_{J_{j}}2^{-i\left( 1-p/2\right) }\Biggl( 2^{s+i}2^i 2^i 2^j 2^j\Biggr)^{p/2} d\mu(y) =  c\int\limits_{J_{j}}2^{-i\left( 1-p/2\right) }\Biggl( 2^{s+i}2^{2i+2j+s}\Biggr)^{p/2} d\mu(y).
\end{split}
\end{equation*}

This gives that for $A_{2}$ we
obtain%

\begin{eqnarray}
A_{2} &\leq
&c_{p}\sum\limits_{i=0}^{N}\sum\limits_{j=i}^{N}\sum\limits_{s=i}^{N}\int%
\limits_{J_{j}}2^{-i\left( 1-p/2\right) }\left( 2^{2i+2j+s}\right)
^{p/2}d\mu \left( y\right)  \label{A2_1} \\
&\leq &c_{p}2^{Np/2}\sum\limits_{i=0}^{N}2^{i\left( 3p/2-1\right)
}\sum\limits_{j=i}^{N}2^{j\left( p-1\right) }  \notag \\
&\leq &c_{p}2^{Np/2}\sum\limits_{i=0}^{N}2^{i\left( 5p/2-2\right) }  \notag
\\
&\leq &c_{p}2^{N\left( 3p-2\right) },  \notag
\end{eqnarray}%
for $p\in \left( 4/5,1\right)$. Let $p=1$. Then%
\begin{eqnarray}
A_{2} &\leq &2^{N/2}\sum\limits_{i=0}^{N}2^{i/2}\left( N-i+1\right)
\label{A2_2} \\
&\leq &2^{N}\sum\limits_{i=0}^{N}\frac{\left( N-i+1\right) }{2^{\left(
N-i\right) /2}}\leq c2^{N}.  \notag
\end{eqnarray}

Combining (\ref{A+B})-(\ref{A2_2}) we complete the proof of Lemma \ref%
{lemma2}.
\end{proof}

\begin{corollary}
\label{cor1}Let $\alpha _{1}\left( n,k\right) =k,\alpha _{2}\left(
n,k\right) =n-k,k=0,...,n-1,n\in \mathbb{P}$, $p\in (4/5,1].$ Then%
\begin{equation*}
\sup\limits_{1\leq n\leq 2^{N}}\int\limits_{\overline{I}_{N}\times \overline{%
I}_{N}}\left\vert \sum\limits_{k=0}^{n-1}D_{k}\left( x\right) D_{n-k}\left(
y\right) \right\vert ^{p}d\mu \left( x,y\right) \leq c_{p}2^{N\left(
3p-2\right) }.
\end{equation*}
\end{corollary}

\begin{corollary}
\label{cor2}Let $p\in (4/5,1]$. Then
\begin{equation*}
\sup\limits_{0\leq q<2^{N}}\int\limits_{\overline{I}_{N}\times \overline{I}%
_{N}}\left\vert \sum\limits_{k=q}^{2^{N}-1}D_{k}\left( x\right)
D_{k-q}\left( y\right) \right\vert ^{p}d\mu \left( x,y\right) \leq
c_{p}2^{N\left( 3p-2\right) }.
\end{equation*}
\end{corollary}

\begin{proof}[Proof of Corollary \protect\ref{cor2}]
Since
\begin{equation*}
w_{2^{N}-1}D_{j}=D_{2^{N}}-D_{2^{N}-j},j=0,1,...,2^{N}-1
\end{equation*}%
from Lemma \ref{sup1} and Corollary \ref{cor1} we can write%
\begin{equation*}
\int\limits_{\overline{I}_{N}\times \overline{I}_{N}}\left\vert
\sum\limits_{k=q}^{2^{N}-1}D_{k}\left( x\right) D_{k-q}\left( y\right)
\right\vert ^{p}d\mu \left( x,y\right)
\end{equation*}%
\begin{equation*}
\leq \int\limits_{\overline{I}_{N}\times \overline{I}_{N}}\left\vert
\sum\limits_{k=0}^{2^{N}-q}\left( D_{2^{N}}\left( x\right) -D_{2^{N}-\left(
k+q\right) }\left( x\right) \right) D_{k}\left( y\right) \right\vert
^{p}d\mu \left( x,y\right)
\end{equation*}%
\begin{equation*}
\leq c_{p}\int\limits_{\overline{I}_{N}\times \overline{I}%
_{N}}D_{2^{N}}\left( x\right) \left\vert
\sum\limits_{k=0}^{2^{N}-q}D_{k}\left( y\right) \right\vert ^{p}d\mu \left(
x,y\right)
\end{equation*}%
\begin{equation*}
+c_{p}\int\limits_{\overline{I}_{N}\times \overline{I}_{N}}\left\vert
\sum\limits_{k=0}^{2^{N}-q}D_{2^{N}-q-k}\left( x\right) D_{k}\left( y\right)
\right\vert ^{p}d\mu \left( x,y\right) \leq c_{p}2^{N\left( 3p-2\right) }.
\end{equation*}

Corollary \ref{cor2} is proved.
\end{proof}
It is well known that $D_0=0$ but to have symmetry, that is for some technical reason,  at same places we have to indicate the $0$th Dirichlet kernel in sums below.
\begin{lemma}
\label{lemma3}Let $n=2^{N}q_{1}+q_{2},0\leq q_{2}<2^{N}$. Then%
\begin{equation*}
K_{n}^{\vartriangle }\left( x,y\right)
=\sum\limits_{l=0}^{q_{1}-1}w_{l}\left( 2^{N}x\right) w_{q_{1}-l}\left(
2^{N}y\right) \sum\limits_{k=1}^{q_{2}-1}D_{k}\left( x\right)
D_{q_{2}-k}\left( y\right)
\end{equation*}%
\begin{equation*}
+D_{2^{N}}\left( y\right) \sum\limits_{l=0}^{q_{1}-1}w_{l}\left(
2^{N}x\right) D_{q_{1}-l}\left( 2^{N}y\right)
\sum\limits_{k=1}^{2^{N}}D_{k}\left( x\right)
\end{equation*}%
\begin{equation*}
+D_{2^{N}}\left( x\right) \sum\limits_{l=0}^{q_{1}-1}D_{l}\left(
2^{N}x\right) w_{q_{1}-l}\left( 2^{N}y\right) q_{2}K_{q_{2}}\left( y\right)
\end{equation*}%
\begin{equation*}
+D_{2^{N}}\left( x\right) D_{2^{N}}\left( y\right)
\sum\limits_{l=0}^{q_{1}-1}D_{l}\left( 2^{N}x\right) D_{q_{1}-l}\left(
2^{N}y\right) 2^{N}
\end{equation*}%
\begin{equation*}
-w_{2^{N}-1}\left( y\right) \sum\limits_{l=0}^{q_{1}-1}w_{l}\left(
2^{N}x\right) w_{q_{1}-l-1}\left( 2^{N}y\right)
\sum\limits_{k=q_{2}+1}^{2^{N}}D_{k}\left( x\right) D_{k-q_{2}-1}\left(
y\right)
\end{equation*}%
\begin{equation*}
-w_{2^{N}-1}\left( y\right) D_{2^{N}}\left( x\right)
\sum\limits_{l=0}^{q_{1}-1}D_{l}\left( 2^{N}x\right) w_{q_{1}-l-1}\left(
2^{N}y\right) \sum\limits_{k=q_{2}+1}^{2^{N}}D_{k-q_{2}}\left( y\right)
\end{equation*}%
\begin{equation*}
+w_{q_{1}}\left( 2^{N}x\right) \sum\limits_{k=1}^{q_{2}-1}D_{k}\left(
x\right) D_{q_{2}-k}\left( y\right) +D_{2^{N}}\left( x\right)
D_{q_{1}}\left( 2^{N}x\right) \sum\limits_{v=1}^{q_{2}-1}D_{v}\left(
y\right).
\end{equation*}
\end{lemma}

\begin{proof}[Proof of Lemma \protect\ref{lemma3}]
We can write%
\begin{eqnarray}
K_{n}^{\vartriangle }\left( x,y\right)
&=&\sum\limits_{k=1}^{2^{N}q_{1}+q_{2}-1}D_{k}\left( x\right) D_{n-k}\left(
y\right)  \label{I1+I2} \\
&=&\sum\limits_{k=1}^{2^{N}q_{1}}D_{k}\left( x\right) D_{n-k}\left( y\right)
+\sum\limits_{k=1}^{q_{2}-1}D_{k+2^{N}q_{1}}\left( x\right)
D_{q_{2}-k}\left( y\right)  \notag \\
&=&I_{1}+I_{2}.  \notag
\end{eqnarray}

Since (see \cite[(6.1)]{fin})
\begin{equation}
D_{k+l2^{N}}\left( x\right) =w_{l}\left( 2^{N}x\right) D_{k}\left( x\right)
+D_{2^{N}}\left( x\right) D_{l}\left( 2^{N}x\right)  \label{D}
\end{equation}%
for $I_{2}$ we have%
\begin{eqnarray}
I_{2} &=&w_{q_{1}}\left( 2^{N}x\right)
\sum\limits_{k=1}^{q_{2}-1}D_{k}\left( x\right) D_{q_{2}-k}\left( y\right)
\label{I2} \\
&&+D_{2^{N}}\left( x\right) D_{q_{1}}\left( 2^{N}x\right)
\sum\limits_{k=1}^{q_{2}-1}D_{k}\left( y\right)  \notag
\end{eqnarray}

For $I_{1}$ we have%
\begin{eqnarray}
I_{1}
&=&\sum\limits_{l=0}^{q_{1}-1}\sum\limits_{k=1}^{2^{N}}D_{k+l2^{N}}\left(
x\right) D_{2^{N}q_{1}+q_{2}-l2^{N}-k}\left( y\right)  \label{I11+I12} \\
&=&\sum\limits_{l=0}^{q_{1}-1}\sum\limits_{k=1}^{q_{2}}D_{k+l2^{N}}\left(
x\right) D_{2^{N}\left( q_{1}-l\right) +q_{2}-k}\left( y\right)  \notag \\
&&+\sum\limits_{l=0}^{q_{1}-1}\sum\limits_{k=q_{2}+1}^{2^{N}}D_{k+l2^{N}}%
\left( x\right) D_{2^{N}\left( q_{1}-l\right) +q_{2}-k}\left( y\right)
\notag \\
&=&I_{11}+I_{12}.  \notag
\end{eqnarray}

From (\ref{D}) we get%
\begin{eqnarray}
I_{11} &=&\sum\limits_{l=0}^{q_{1}-1}w_{l}\left( 2^{N}x\right)
w_{q_{1-l}}\left( 2^{N}y\right) \sum\limits_{k=1}^{q_{2}}D_{k}\left(
x\right) D_{q_{2}-k}\left( y\right)  \label{I11} \\
&&+D_{2^{N}}\left( y\right) \sum\limits_{l=0}^{q_{1}-1}w_{l}\left(
2^{N}x\right) D_{q_{1-l}}\left( 2^{N}y\right)
\sum\limits_{k=1}^{q_{2}}D_{k}\left( x\right)  \notag \\
&&+D_{2^{N}}\left( x\right) \sum\limits_{l=0}^{q_{1}-1}D_{l}\left(
2^{N}x\right) w_{q_{1-l}}\left( 2^{N}y\right) q_{2}K_{q_{2}}\left( y\right)
\notag \\
&&+D_{2^{N}}\left( x\right) D_{2^{N}}\left( y\right)
\sum\limits_{l=0}^{q_{1}-1}D_{l}\left( 2^{N}x\right) D_{q_{1-l}}\left(
2^{N}y\right) q_{2}.  \notag
\end{eqnarray}

Turn our attention to $I_{12}$. From the simple calculation we can write ($%
q_{2}<k\leq 2^{N},0\leq l<q_{1}$)%
\begin{eqnarray*}
&&D_{2^{N}\left( q_{1}-l\right) -(k-q_{2})}\left( y\right) \\
&=&D_{2^{N}\left( q_{1}-l-1\right) +2^{N}-(k-q_{2})}\left( y\right) \\
&=&D_{2^{N}\left( q_{1}-l-1\right) }\left( y\right) +w_{2^{N}\left(
q_{1}-l-1\right) }\left( y\right) D_{2^{N}-(k-q_{2})}\left( y\right).
\end{eqnarray*}%
Since%
\begin{equation*}
D_{2^{N}-(k-q_{2})}\left( y\right) =D_{2^{N}}\left( y\right)
-w_{2^{N}-1}\left( y\right) D_{k-q_{2}}\left( y\right),
\end{equation*}%
we obtain%
\begin{eqnarray*}
&&D_{2^{N}\left( q_{1}-l\right) +q_{2}-k}\left( y\right) \\
&=&D_{2^{N}}\left( y\right) D_{q_{1}-l-1}\left( 2^{N}y\right)
+w_{q_{1}-l-1}\left( 2^{N}y\right) D_{2^{N}}\left( y\right) \\
&&-w_{q_{1}-l-1}\left( 2^{N}y\right) w_{2^{N}-1}\left( y\right)
D_{k-q_{2}}\left( y\right).
\end{eqnarray*}%
Hence from (\ref{D}), we have%
\begin{equation}
I_{12}=D_{2^{N}}\left( y\right) \sum\limits_{l=0}^{q_{1}-1}w_{l}\left(
2^{N}x\right) D_{q_{1}-l-1}\left( 2^{N}y\right)
\sum\limits_{k=q_{2}+1}^{2^{N}}D_{k}\left( x\right)  \label{I12}
\end{equation}%
\begin{equation*}
+D_{2^{N}}\left( y\right) D_{2^{N}}\left( x\right)
\sum\limits_{l=0}^{q_{1}-1}D_{l}\left( 2^{N}x\right) D_{q_{1}-l-1}\left(
2^{N}y\right) \left( 2^{N}-q_{2}\right)
\end{equation*}%
\begin{equation*}
+D_{2^{N}}\left( y\right) \sum\limits_{l=0}^{q_{1}-1}w_{l}\left(
2^{N}x\right) w_{q_{1}-l-1}\left( 2^{N}y\right)
\sum\limits_{k=q_{2}+1}^{2^{N}}D_{k}\left( x\right)
\end{equation*}%
\begin{equation*}
+D_{2^{N}}\left( y\right) D_{2^{N}}\left( x\right)
\sum\limits_{l=0}^{q_{1}-1}D_{l}\left( 2^{N}x\right) w_{q_{1}-l-1}\left(
2^{N}y\right) \left( 2^{N}-q_{2}\right)
\end{equation*}%
\begin{equation*}
-w_{2^{N}-1}\left( y\right) \sum\limits_{l=0}^{q_{1}-1}w_{l}\left(
2^{N}x\right) w_{q_{1}-l-1}\left( 2^{N}y\right)
\sum\limits_{k=q_{2}+1}^{2^{N}}D_{k}\left( x\right) D_{k-q_{2}}\left(
y\right)
\end{equation*}%
\begin{equation*}
-w_{2^{N}-1}\left( y\right) D_{2^{N}}\left( x\right)
\sum\limits_{l=0}^{q_{1}-1}D_{l}\left( 2^{N}x\right) w_{q_{1}-l-1}\left(
2^{N}y\right) \sum\limits_{k=q_{2}+1}^{2^{N}}D_{k-q_{2}}\left( y\right)
\end{equation*}%
\begin{equation*}
=D_{2^{N}}\left( y\right) \sum\limits_{l=0}^{q_{1}-1}w_{l}\left(
2^{N}x\right) D_{q_{1}-l}\left( 2^{N}y\right)
\sum\limits_{k=q_{2}+1}^{2^{N}}D_{k}\left( x\right)
\end{equation*}%
\begin{equation*}
+D_{2^{N}}\left( y\right) D_{2^{N}}\left( x\right)
\sum\limits_{l=0}^{q_{1}-1}D_{l}\left( 2^{N}x\right) D_{q_{1}-l}\left(
2^{N}y\right) \left( 2^{N}-q_{2}\right)
\end{equation*}%
\begin{equation*}
-w_{2^{N}-1}\left( y\right) \sum\limits_{l=0}^{q_{1}-1}w_{l}\left(
2^{N}x\right) w_{q_{1}-l-1}\left( 2^{N}y\right)
\sum\limits_{k=q_{2}+1}^{2^{N}}D_{k}\left( x\right) D_{k-q_{2}}\left(
y\right)
\end{equation*}%
\begin{equation*}
-w_{2^{N}-1}\left( y\right) D_{2^{N}}\left( x\right)
\sum\limits_{l=0}^{q_{1}-1}D_{l}\left( 2^{N}x\right) w_{q_{1}-l-1}\left(
2^{N}y\right) \sum\limits_{k=q_{2}+1}^{2^{N}}D_{k-q_{2}}\left( y\right) .
\end{equation*}

Combining (\ref{I1+I2}), (\ref{I11+I12}), (\ref{I11}) and (\ref{I12}) we
complete the proof of Lemma \ref{lemma3}.
\end{proof}

\begin{lemma}
\label{lemma4}Let $\left( x,y\right) \in \overline{I}_{N}\times \overline{I}%
_{N}$ and $n=2^{N}q_{1}+q_{2},0\leq q_{2}<2^{N}$. Then the following
inequality holds%
\begin{eqnarray*}
&&\int\limits_{I_{N}\times I_{N}}\left\vert K_{n}^{\bigtriangleup }\left(
x+s,y+t\right) \right\vert d\mu \left( s,t\right) \\
&\leq &\frac{c}{2^{3N}}\left\{ \left\vert
\sum\limits_{k=0}^{q_{2}-1}D_{k}\left( x\right) D_{q_{2}-k}\left( y\right)
\right\vert +\left\vert \sum\limits_{k=q_{2}+1}^{2^{N}}D_{k}\left( x\right)
D_{k-q_{2}}\left( y\right) \right\vert \right\} .
\end{eqnarray*}
\end{lemma}

\begin{proof}[Proof of Lemma \protect\ref{lemma4}]
Since $x+s,y+t\notin I_{N}$ , by lemma \ref{lemma3} we have%
\begin{eqnarray*}
&&nK_{n}^{\bigtriangleup }\left( x,y\right) \\
&=&\sum\limits_{l=0}^{q_{1}-1}w_{l}\left( 2^{N}x\right) w_{q_{1}-l}\left(
2^{N}y\right) \sum\limits_{k=0}^{q_{2}-1}D_{k}\left( x\right)
D_{q_{2}-k}\left( y\right) \\
&&-w_{2^{N}-1}\left( y\right) \sum\limits_{l=0}^{q_{1}-1}w_{l}\left(
2^{N}x\right) w_{q_{1}-l-1}\left( 2^{N}y\right)
\sum\limits_{k=q_{2}+1}^{2^{N}}D_{k}\left( x\right) D_{k-q_{2}}\left(
y\right) \\
&&+w_{q_{1}}\left( 2^{N}x\right) \sum\limits_{k=0}^{q_{2}-1}D_{k}\left(
x\right) D_{q_{2}-k}\left( y\right) .
\end{eqnarray*}%
Consequently $\left( q_{1}\leq n2^{-N}\right) $%
\begin{equation*}
\int\limits_{I_{N}\times I_{N}}\left\vert K_{n}^{\bigtriangleup }\left(
x+s,y+t\right) \right\vert d\mu \left( s,t\right)
\end{equation*}%
\begin{equation*}
\leq \frac{1}{n}\left\vert \sum\limits_{k=0}^{q_{2}-1}D_{k}\left( x\right)
D_{q_{2}-k}\left( y\right) \right\vert
\end{equation*}%
\begin{equation*}
\times \int\limits_{I_{N}\times I_{N}}\left\vert
\sum\limits_{l=0}^{q_{1}-1}w_{l}\left( 2^{N}\left( x+s\right) \right)
w_{q_{1}-l}\left( 2^{N}\left( y+t\right) \right) \right\vert
\end{equation*}%
\begin{equation*}
+\frac{1}{n}\left\vert \sum\limits_{k=q_{2}+1}^{2^{N}}D_{k}\left( x\right)
D_{k-q_{2}}\left( y\right) \right\vert
\end{equation*}%
\begin{equation*}
\times \int\limits_{I_{N}\times I_{N}}\left\vert
\sum\limits_{l=0}^{q_{1}-1}w_{l}\left( 2^{N}\left( x+s\right) \right)
w_{q_{1}-l-1}\left( 2^{N}\left( y+t\right) \right) \right\vert
\end{equation*}%
\begin{equation*}
+\frac{1}{n2^{2N}}\left\vert \sum\limits_{k=0}^{q_{2}-1}D_{k}\left( x\right)
D_{q_{2}-k}\left( y\right) \right\vert
\end{equation*}%
\begin{equation*}
\leq \frac{cq_{1}}{n2^{2N}}\left\{ \left\vert
\sum\limits_{k=0}^{q_{2}-1}D_{k}\left( x\right) D_{q_{2}-k}\left( y\right)
\right\vert +\left\vert \sum\limits_{k=q_{2}+1}^{2^{N}}D_{k}\left( x\right)
D_{k-q_{2}}\left( y\right) \right\vert \right\}
\end{equation*}%
\begin{equation*}
\leq \frac{c}{2^{3N}}\left\{ \left\vert
\sum\limits_{k=0}^{q_{2}-1}D_{k}\left( x\right) D_{q_{2}-k}\left( y\right)
\right\vert +\left\vert \sum\limits_{k=q_{2}+1}^{2^{N}-1}D_{k}\left(
x\right) D_{k-q_{2}}\left( y\right) \right\vert \right\} .
\end{equation*}

Lemma \ref{lemma4} is proved.
\end{proof}

\begin{lemma}
\label{lemma5}Let $\left( x,y\right) \in \overline{I}_{N}\times I_{N}$ and $%
n=2^{N}q_{1}+q_{2},0\leq q_{2}<2^{N}$. Then the following inequality holds%
\begin{eqnarray*}
&&\int\limits_{I_{N}\times I_{N}}\left\vert K_{n}^{\bigtriangleup }\left(
x+s,y+\dotplus t\right) \right\vert d\mu \left( s,t\right) \\
&\leq &\frac{c}{2^{3N}}\left\{ \left\vert
\sum\limits_{k=0}^{q_{2}-1}D_{k}\left( x\right) D_{q_{2}-k}\left( y\right)
\right\vert +\left\vert \sum\limits_{k=q_{2}+1}^{2^{N}}D_{k}\left( x\right)
D_{k-q_{2}}\left( y\right) \right\vert \right. \\
&&\left. +D_{2^{N}}\left( y\right) \left\vert
\sum\limits_{k=1}^{2^{N}}D_{k}\left( x\right) \right\vert \right\} ,n\geq
2^{N}.
\end{eqnarray*}
\end{lemma}

\begin{proof}[Proof of Lemma \protect\ref{lemma5}]
From (\ref{dir}) and Lemma \ref{lemma3} we get%
\begin{eqnarray*}
&&nK_{n}^{\bigtriangleup }\left( x,y\right) \\
&=&\sum\limits_{l=0}^{q_{1}-1}w_{l}\left( 2^{N}x\right) w_{q_{1}-l}\left(
2^{N}y\right) \sum\limits_{k=0}^{q_{2}-1}D_{k}\left( x\right)
D_{q_{2}-k}\left( y\right) \\
&&+\left( \sum\limits_{k=1}^{2^{N}-1}D_{k}\left( x\right) \right)
D_{2^{N}}\left( y\right) \sum\limits_{l=0}^{q_{1}-1}w_{l}\left(
2^{N}x\right) D_{q_{1}-l}\left( 2^{N}y\right) \\
&&-w_{2^{N}-1}\left( y\right) \sum\limits_{l=0}^{q_{1}-1}w_{l}\left(
2^{N}x\right) w_{q_{1}-l-1}\left( 2^{N}y\right)
\sum\limits_{k=q_{2}+1}^{2^{N}}D_{k}\left( x\right) D_{k-q_{2}}\left(
y\right) \\
&&+w_{q_{1}}\left( 2^{N}x\right) \sum\limits_{k=0}^{q_{2}-1}D_{k}\left(
x\right) D_{q_{2}-k}\left( y\right) .
\end{eqnarray*}

Consequently,
\begin{equation}
\int\limits_{I_{N}\times I_{N}}\left\vert K_{n}^{\bigtriangleup }\left(
x+s,y+t\right) \right\vert d\mu \left( s,t\right)   \label{l5}
\end{equation}%
\begin{equation*}
\leq \frac{1}{n}\left\vert \sum\limits_{k=0}^{q_{2}-1}D_{k}\left( x\right)
D_{q_{2}-k}\left( y\right) \right\vert
\end{equation*}%
\begin{equation*}
\times \int\limits_{I_{N}\times I_{N}}\left\vert
\sum\limits_{l=0}^{q_{1}-1}w_{l}\left( 2^{N}\left( x+s\right) \right)
w_{q_{1}-l}\left( 2^{N}\left( y+t\right) \right) \right\vert d\mu \left(
s,t\right)
\end{equation*}%
\begin{equation*}
+\frac{D_{2^{N}}\left( y\right) }{n}\left\vert
\sum\limits_{k=1}^{2^{N}-1}D_{k}\left( x\right) \right\vert
\end{equation*}%
\begin{equation*}
\times \int\limits_{I_{N}\times I_{N}}\left\vert
\sum\limits_{l=0}^{q_{1}-1}w_{l}\left( 2^{N}\left( x+s\right) \right)
D_{q_{1}-l}\left( 2^{N}\left( y+t\right) \right) \right\vert d\mu \left(
s,t\right)
\end{equation*}%
\begin{equation*}
+\frac{1}{n}\left\vert \sum\limits_{k=q_{2}}^{2^{N}-1}D_{k}\left( x\right)
D_{k-q_{2}}\left( y\right) \right\vert
\end{equation*}%
\begin{equation*}
\times \int\limits_{I_{N}\times I_{N}}\left\vert
\sum\limits_{l=0}^{q_{1}-1}w_{l}\left( 2^{N}\left( x\dotplus s\right)
x\right) w_{q_{1}-l-1}\left( 2^{N}\left( y\dotplus t\right) \right)
\right\vert d\mu \left( s,t\right)
\end{equation*}%
\begin{equation*}
+\frac{1}{2^{3N}}\left\vert \sum\limits_{k=0}^{q_{2}-1}D_{k}\left( x\right)
D_{q_{2}-k}\left( y\right) \right\vert
\end{equation*}%
\begin{equation*}
\leq \frac{q_{1}}{n2^{2N}}\left\vert \sum\limits_{k=0}^{q_{2}-1}D_{k}\left(
x\right) D_{q_{2}-k}\left( y\right) \right\vert
\end{equation*}%
\begin{equation*}
+\frac{D_{2^{N}}\left( y\right) }{n2^{2N}}\left\vert
\sum\limits_{k=1}^{2^{N}-1}D_{k}\left( x\right) \right\vert
\int\limits_{G\times G}\left\vert \sum\limits_{l=0}^{q_{1}-1}w_{l}\left(
u\right) D_{q_{1}-l}\left( v\right) \right\vert d\mu \left( u,v\right)
\end{equation*}%
\begin{equation*}
+\frac{q_{1}}{n2^{2N}}\left\vert \sum\limits_{k=q_{2}+1}^{2^{N}}D_{k}\left(
x\right) D_{k-q_{2}}\left( y\right) \right\vert
\end{equation*}%
\begin{equation*}
+\frac{1}{2^{3N}}\left\vert \sum\limits_{k=0}^{q_{2}-1}D_{k}\left( x\right)
D_{q_{2}-k}\left( y\right) \right\vert .
\end{equation*}%
Since
\begin{eqnarray*}
&&\int\limits_{G\times G}\left\vert \sum\limits_{l=0}^{q_{1}-1}w_{l}\left(
u\right) D_{q_{1}-l}\left( v\right) \right\vert d\mu \left( u,v\right)  \\
&\leq &\left( \int\limits_{G\times G}\left\vert
\sum\limits_{l=0}^{q_{1}-1}w_{l}\left( u\right) D_{q_{1}-l}\left( v\right)
\right\vert ^{2}d\mu \left( u,v\right) \right) ^{1/2} \\
&=&\left(\int\limits_{G}\sum\limits_{l=0}^{q_{1}-1}D_{q_{1}-l}^{2}\left( v\right)
d\mu \left( v\right) \right)^{1/2} \\
&=&\sum\limits_{l=0}^{q_{1}-1}\left( q_{1}-l\right) \leq q_{1},
\end{eqnarray*}%
by (\ref{l5}) we complete the proof of Lemma \ref{lemma5}.
\end{proof}

\section{Proofs of main results}

\begin{proof}[Proof of Theorem \protect\ref{main}]
By Lemma \ref{weisz}, the proof of theorem will be complete if we show that
the operator $\sigma _{n}^{\bigtriangleup }$ is $p$-quasi-local for each $%
4/5<p\leq 1$ and bounded from $L_{\infty }\left( G\times G\right) $ to $%
L_{\infty }\left( G\times G\right) $ $.$

First, we prove the boundedness from $L_{\infty }\left( G\times G\right) $
to $L_{\infty }\left( G\times G\right) $.  Ee can write%
\begin{equation}
\int\limits_{G\times G}\frac{1}{n}\left\vert
\sum\limits_{k=1}^{n-1}D_{k}\left( x\right) D_{n-k}\left( y\right)
\right\vert d\mu \left( x,y\right)  \label{D1-D4}
\end{equation}%
\begin{equation*}
=\int\limits_{\overline{I}_{|n|}\times \overline{I}_{|n|}}\frac{1}{n}%
\left\vert \sum\limits_{k=1}^{n-1}D_{k}\left( x\right) D_{n-k}\left(
y\right) \right\vert d\mu \left( x,y\right)
\end{equation*}%
\begin{equation*}
+\int\limits_{\overline{I}_{|n|}\times I_{|n|}}\frac{1}{n}\left\vert
\sum\limits_{k=1}^{n-1}D_{k}\left( x\right) D_{n-k}\left( y\right)
\right\vert d\mu \left( x,y\right)
\end{equation*}%
\begin{equation*}
+\int\limits_{I_{|n|}\times \overline{I}_{|n|}}\frac{1}{n}\left\vert
\sum\limits_{k=1}^{n-1}D_{k}\left( x\right) D_{n-k}\left( y\right)
\right\vert d\mu \left( x,y\right)
\end{equation*}%
\begin{equation*}
+\int\limits_{I_{|n|}\times I_{|n|}}\frac{1}{n}\left\vert
\sum\limits_{k=1}^{n-1}D_{k}\left( x\right) D_{n-k}\left( y\right)
\right\vert d\mu \left( x,y\right)
\end{equation*}

From Corollary \ref{cor1} we have%
\begin{equation}
\sup\limits_{n}\int\limits_{\overline{I}_{N}\times \overline{I}_{N}}\frac{1}{%
n}\left\vert \sum\limits_{k=0}^{n-1}D_{k}\left( x\right) D_{n-k}\left(
y\right) \right\vert d\mu \left( x,y\right) <\infty .  \label{1}
\end{equation}

Using Lemma \ref{lemma1} for $p=1$ we have%
\begin{equation}
\int\limits_{\overline{I}_{|n|}\times I_{|n|}}\frac{1}{n}\left\vert
\sum\limits_{k=1}^{n-1}D_{k}\left( x\right) D_{n-k}\left( y\right)
\right\vert d\mu \left( x,y\right)  \label{2}
\end{equation}%
\begin{equation*}
=\frac{1}{n2^{|n|}}\int\limits_{G}\left\vert \sum\limits_{k=1}^{n-1}\left(
n-k\right) D_{k}\left( x\right) \right\vert d\mu \left( x,y\right) \leq
c<\infty ,n\in \mathbb{P}.
\end{equation*}

Analogously, we can prove that%
\begin{equation}
\sup\limits_{n}\int\limits_{I_{|n|}\times \overline{I}_{|n|}}\frac{1}{n}%
\left\vert \sum\limits_{k=1}^{n-1}D_{k}\left( x\right) D_{n-k}\left(
y\right) \right\vert d\mu \left( x,y\right) <\infty .  \label{3}
\end{equation}

From a simple calculation we have%
\begin{eqnarray}
&&\int\limits_{I_{|n|}\times I_{|n|}}\frac{1}{n}\left\vert
\sum\limits_{k=1}^{n-1}D_{k}\left( x\right) D_{n-k}\left( y\right)
\right\vert d\mu \left( x,y\right)  \label{4} \\
&\leq &\frac{1}{2^{3N}}\sum\limits_{k=1}^{n-1}k\left( n-k\right) \leq
c<\infty .  \notag
\end{eqnarray}

Combining (\ref{D1-D4})-(\ref{4}) we conclude that%
\begin{equation*}
\sup\limits_{n}\int\limits_{G\times G}\frac{1}{n}\left\vert
\sum\limits_{k=1}^{n-1}D_{k}\left( x\right) D_{n-k}\left( y\right)
\right\vert d\mu \left( x,y\right) <\infty .
\end{equation*}

Hence, from\ (\ref{triang}) and (\ref{tr-kernel}) we conclude that%
\begin{equation*}
\left\Vert \sigma _{n }^{\bigtriangleup }f\right\Vert _{\infty }\leq
c\left\Vert f\right\Vert _{\infty }.
\end{equation*}

Now, we prove that $\sigma _{n}^{\bigtriangleup }$ is $p$-quasi-local. We
considere three casses. Let $a$ be an arbitrary atom with support $%
I_{N}(u)\times I_{N}(v)$. It is easy to see that $\sigma
_{n}^{\bigtriangleup }(a)=0$ if $n<2^{N}$. Therefore we can suppose that $%
n\geq 2^{N}$. Since the dyadic addition is a measure preserving group
operation, we may assume that $u=v=0.$

\textbf{Step 1. Integrating over} $\overline{I}_{N}\times \overline{I}_{N}$.
Since $\left\Vert a\right\Vert _{\infty }\leq c^{2N/p}$ from Lemma \ref%
{lemma4} we obtain%
\begin{equation}
\left\vert \sigma _{n}^{\bigtriangleup }a\left( x,y\right) \right\vert \leq
\left\Vert a\right\Vert _{\infty }\int\limits_{I_{N}\times I_{N}}\left\vert
K_{n}^{\bigtriangleup }\left( x+s,y+t\right) \right\vert d\mu \left(
s,t\right)  \label{p1}
\end{equation}%
\begin{equation*}
\leq \frac{c2^{2N/p}}{2^{3N}}\left\vert
\sum\limits_{k=1}^{q_{2}-1}D_{k}\left( x\right) D_{q_{2}-k}\left( y\right)
\right\vert
\end{equation*}%
\begin{equation*}
+\frac{c2^{2N/p}}{2^{3N}}\left\vert
\sum\limits_{k=q_{2}+1}^{2^{N}-1}D_{k}\left( x\right) D_{k-q_{2}}\left(
y\right) \right\vert ,n=2^{N}q_{1}+q_{2},0\leq q_{2}<2^{N}.
\end{equation*}

Applying the inequality%
\begin{equation*}
\left( \sum\limits_{k}a_{k}\right) ^{p}\leq \sum\limits_{k}a_{k}^{p},~\
\left( 0<p\leq 1\right)
\end{equation*}%
by (\ref{p1}) we have%
\begin{eqnarray*}
&&\left\vert \sigma _{n}^{\bigtriangleup }a\left( x,y\right) \right\vert ^{p}
\\
&\leq &\frac{c_{p}2^{2N}}{2^{3Np}}\left\{ \left( \left\vert
\sum\limits_{k=1}^{q_{2}-1}D_{k}\left( x\right) D_{q_{2}-k}\left( y\right)
\right\vert ^{p}+\left\vert \sum\limits_{k=q_{2}+1}^{2^{N}-1}D_{k}\left(
x\right) D_{k-q_{2}}\left( y\right) \right\vert ^{p}\right) \right\} .
\end{eqnarray*}%
Consequently, by Corollary \ref{cor1} \ and Corollary \ref{cor2} we can write%
\begin{eqnarray}
&&\int\limits_{\overline{I}_{N}\times \overline{I}_{N}}\left\vert \sigma
_{n}^{\bigtriangleup }a\left( x,y\right) \right\vert ^{p}d\mu \left(
x,y\right)  \label{main1} \\
&\leq &\frac{c_{p}2^{2N}}{2^{3Np}}\left\{ \sup\limits_{1\leq q_{2}\leq
2^{N}}\int\limits_{\overline{I}_{N}\times \overline{I}_{N}}\left\vert
\sum\limits_{k=1}^{q_{2}-1}D_{k}\left( x\right) D_{q_{2}-k}\left( y\right)
\right\vert ^{p}d\mu \left( x,y\right) \right.  \notag \\
&&\left. +\sup\limits_{1\leq q_{2}\leq 2^{N}}\int\limits_{\overline{I}%
_{N}\times \overline{I}_{N}}\left\vert
\sum\limits_{k=q_{2}+1}^{2^{N}-1}D_{k}\left( x\right) D_{k-q_{2}}\left(
y\right) \right\vert ^{p}d\mu \left( x,y\right) \right\}  \notag \\
&\leq &c_{p}<\infty ,\text{ \ \ }4/5<p\leq 1.  \notag
\end{eqnarray}

\textbf{Step 2. Integrating over} $\overline{I}_{N}\times I_{N}$. From Lemma %
\ref{lemma5} we obtain%
\begin{eqnarray*}
&&\left\vert \sigma _{n}^{\bigtriangleup }a\left( x,y\right) \right\vert \\
&\leq &\left\Vert a\right\Vert _{\infty }\int\limits_{I_{N}\times
I_{N}}\left\vert K_{n}^{\bigtriangleup }\left( x\dotplus s,y\dotplus
t\right) \right\vert d\mu \left( s,t\right) \\
&\leq &\frac{c2^{2N/p}}{2^{3N}}\left\{ \left\vert
\sum\limits_{k=1}^{q_{2}-1}D_{k}\left( x\right) D_{q_{2}-k}\left( y\right)
\right\vert +\left\vert \sum\limits_{k=q_{2}+1}^{2^{N}-1}D_{k}\left( x\right)
D_{k-q_{2}}\left( y\right) \right\vert \right. \\
&&\left. +D_{2^{N}}\left( y\right) \left\vert
\sum\limits_{k=1}^{2^{N}-1}D_{k}\left( x\right) \right\vert \right\} .
\end{eqnarray*}%
Consequently, from Lemma \ref{sup1}, Lemma \ref{lemma1} and Lemma \ref{sup3}%
, we get
\begin{eqnarray}
&&\int\limits_{\overline{I}_{N}\times I_{N}}\left\vert \sigma
_{n}^{\bigtriangleup }a\left( x,y\right) \right\vert ^{p}d\mu \left(
x,y\right)  \label{main2} \\
&\leq &\frac{c_{p}2^{N}}{2^{3Np}}\left\{ \int\limits_{\overline{I}%
_{N}}\sup\limits_{1\leq q_{2}\leq 2^{N}}\left\vert
\sum\limits_{k=1}^{q_{2}-1}D_{k}\left( x\right) \left( q_{2}-k\right)
\right\vert ^{p}d\mu \left( x\right) \right.  \notag \\
&&+\int\limits_{\overline{I}_{N}}\sup\limits_{1\leq q_{2}\leq
2^{N}}\left\vert \sum\limits_{k=q_{2}}^{2^{N}-1}D_{k}\left( x\right) \left(
k-q_{2}\right) \right\vert ^{p}d\mu \left( x\right)  \notag \\
&&\left. +2^{Np}\int\limits_{\overline{I}_{N}}\left\vert
\sum\limits_{k=1}^{2^{N}-1}D_{k}\left( x\right) \right\vert ^{p}d\mu \left(
x\right) \right\}  \notag \\
&\leq &\frac{c_{p}2^{N}}{2^{3Np}}2^{\left( 3p-1\right) N}=c_{p}<\infty \text{
}\left( 4/5<p\leq 1\right) .  \notag
\end{eqnarray}

\textbf{Step 3. Integrating over} $I_{N}\times \overline{I}_{N}$. The case
is analogous to step 2 and we have%
\begin{equation}
\int\limits_{I_{N}\times \overline{I}_{N}}\left\vert \sigma
_{n}^{\bigtriangleup }a\left( x,y\right) \right\vert ^{p}d\mu \left(
x,y\right) \leq c_{p}<\infty ~\text{\ }\left( 4/5<p\leq 1\right) .
\label{main3}
\end{equation}

Combining (\ref{main1}), (\ref{main2}) and (\ref{main3}) we complete the
proof of Theorem \ref{main}.
\end{proof}

\end{document}